\documentclass[11pt]{amsart}
\usepackage{amsmath,amsfonts,amsthm,amssymb}
\usepackage{amscd}

\usepackage{array}
\usepackage[all]{xy}
\usepackage{epsfig}
\usepackage{graphicx}
\usepackage{color}

\DeclareMathOperator{\Id}{Id}

\usepackage{euscript}

\definecolor{white}{rgb}{1,1,1}\long\def\symbolfootnote[#1]#2{\begingroup\def\thefootnote{\fnsymbol{footnote}}\footnote[#1]{#2}\endgroup} 

\newcounter{thecounter}
\numberwithin{thecounter}{section}
\newtheorem{lemma}[thecounter]{Lemma}
\newtheorem{proposition}[thecounter]{Proposition}
\newtheorem{theorem}[thecounter]{Theorem}
\newtheorem{corollary}[thecounter]{Corollary}
\newtheorem{definition}[thecounter]{Definition}
\newtheorem{rem}[thecounter]{Remark}

\newtheorem{prop}[thecounter]{Proposition}

\newcommand{\Z}{{\mathbb{Z}}}

\newcommand{\Mg}{{\mathfrak{g}}}

\begin{document}

\title{Weyl group orbits on Kac--Moody  root systems}
\author{Lisa Carbone}
\author{Alexander Conway}
\author{Walter Freyn}
\author{Diego Penta}

\address{Department of Mathematics, Hill Center, Busch Campus\\
Rutgers, The State University of New Jersey\\
110 Frelinghuysen Rd\\
Piscataway, NJ 08854-8019}
\email{carbonel@math.rutgers.edu}

\address{Department of Mathematics, Hill Center, Busch Campus\\
Rutgers, The State University of New Jersey\\
110 Frelinghuysen Rd\\
Piscataway, NJ 08854-8019}
\email{ajhconway@gmail.com}

\address{Fachbereich Mathematik\\
 TU Darmstadt\\
Schlossgartenstrasse 7\\
64289 Darmstadt, Germany}
\email{walter.freyn@math.tu-darmstadt.de}

 \address{Department of Mathematical Sciences \\
Binghamton University \\
Binghamton, New York 13902-6000 }
\email{diegopenta@gmail.com}

\thanks{The first author was supported in part by NSF grant number DMS--1101282.}


\begin{abstract} Let $\mathcal{D}$ be a Dynkin diagram and let $\Pi=\{\alpha_1,\dots ,\alpha_{\ell}\}$ be the simple roots of the corresponding Kac--Moody  root system. Let $\mathfrak{h}$ denote the Cartan subalgebra, let $W$ denote the Weyl group and let $\Delta$ denote the set of all roots. The action of $W$ on $\mathfrak{h}$, and hence on $\Delta$, is the discretization of the action of the Kac--Moody  algebra. Understanding the orbit structure of $W$ on $\Delta$ is crucial for many physical applications. We show that for $i\neq j$, the simple roots $\alpha_i$ and $\alpha_j$ are in the same $W$--orbit if and only if vertices $i$ and $j$  in the Dynkin diagram corresponding to $\alpha_i$ and $\alpha_j$ are connected by a path consisting only of single edges.  We introduce the notion of `the Cayley graph $\mathcal{P}$ of the Weyl group action on real roots' whose connected components  are in one-to-one correspondence with the disjoint orbits of $W$. For a symmetric hyperbolic generalized Cartan matrix $A$ of rank $\geq 4$ we prove that any 2 real roots of the same length lie in the same $W$--orbit. We show that if the generalized Cartan matrix $A$ contains zeros, then there are simple roots that are stabilized by simple root reflections in $W$, that is, $W$ does not act simply transitively on  real roots.  We give sufficient conditions in terms of the generalized Cartan matrix $A$ (equivalently ${\mathcal D}$) for $W$ to stabilize a real root. Using symmetry properties of the imaginary light cone in the hyperbolic case, we deduce that the number of $W$--orbits on imaginary roots on a hyperboloid of fixed radius is bounded above by the number of root lattice points on the hyperboloid that intersect the closure of the fundamental region for $W$.

\end{abstract}

\maketitle

\section{Introduction}

\noindent Kac--Moody  algebras are the most natural generalization to infinite dimensions of the notion of a finite dimensional semisimple Lie algebra. The data for constructing a Kac--Moody  algebra contains an integer matrix called a generalized Cartan matrix. Let $A$ be an $\ell\times \ell$ generalized Cartan matrix with Kac--Moody  algebra $\mathfrak{g}=\mathfrak{g}(A)$, root system $\Delta=\Delta(A)$, real roots $\Phi\subseteq\Delta$, and Weyl group $W=W(A)$. Let $\Pi=\{\alpha_1,\ldots,\alpha_{\ell}\}$ denote a fixed base of $\Delta$.

The group $W$ acts on $\Delta$, but a certain subset of $\Delta$ can be described purely in terms of the action of $W$ on $\Pi$.  This is the subset of `real roots' $\Phi=W\Pi$, characterized by the property that their `squared lengths' with respect to a symmetric invariant bilinear form is positive.  This is in contrast to the so-called `imaginary roots' in $\Delta$ whose squared length is zero or negative.  If $A$ is a Cartan matrix of finite type, then all roots are real and so $\Delta=W\Pi$ for a root basis $\Pi$, but if $A$ is not of finite type, then $\Delta$ has a non-trivial subset of imaginary roots.

Our object of study is the orbit structure of $W$ on the  roots of an arbitrary Kac--Moody  root system with Dynkin diagram $\mathcal{D}$ and simple roots $\Pi=\{\alpha_1,\dots ,\alpha_{\ell}\}$.  

We pay special attention to the real and imaginary roots of hyperbolic Kac--Moody  root systems and their images under the Weyl group. These are known to have a physical interpretation. For example, in cosmological billiards, the walls of the billiard table are related to physical fluxes that, in turn, are related to real roots ([DHN]).  In [BGH]  Brown, Ganor and Helfgott show that real roots correspond to fluxes or instantons, and imaginary roots correspond to particles and branes. In [EHTW] Englert,  Houart,  Taormina, and West give a physical interpretation of Weyl group reflections in terms of M--theory. Many physical applications require a full understanding of the structure of the $W$--orbits on imaginary roots.

For an arbitrary Kac--Moody  root system, we show that for $i\neq j$, the simple roots $\alpha_i$ and $\alpha_j$ are in the same $W$--orbit if and only if vertices $i$ and $j$  in the Dynkin diagram corresponding to $\alpha_i$ and $\alpha_j$ are connected by a path consisting only of single edges.  It follows that the disjoint orbits of the Weyl group on real roots are in one-to-one correspondence with connected components of a graph obtained from the Dynkin diagram by removing all multi-edges, arrows and edge-labels. Applying the classification of hyperbolic Dynkin diagrams ([CCCMNNP]), it follows that the maximal number of disjoint Weyl group orbits on real roots of a  hyperbolic root system is 4. 

We introduce the notion of `the Cayley graph $\mathcal{P}$ of the Weyl group action on real roots' which we associate to the Dynkin diagram, or equivalently to the generalized Cartan matrix. The connected components of $\mathcal{P}$ are in one--to--one correspondence with the disjoint orbits of $W$. The vertices of $\mathcal{P}$ in a given connected component are in one--to--one  correspondence with the roots in the corresponding Weyl group orbit and $W$ acts transitively on  each connected component of $\mathcal{P}$.  In [CCCMNNP], we applied this construction to  hyperbolic root systems to determine the disjoint orbits of the action of the Weyl group on real roots.

It should be noted that for a finite root system associated to a simple Lie algebra, all roots of the same length lie in the same $W$--orbit.  This result does not extend to infinite dimensional root systems.  In particular, given a Kac--Moody  root system, real roots of the same length may lie in different $W-$orbits. Given a symmetric generalized Cartan matrix $A$ of  noncompact hyperbolic type such that real roots of the same length lie in distinct orbits, we show that there is a nonsymmetric generalized Cartan matrix $A'$ with the same Weyl group as $A$ and the same Coxeter matrix as $A$, such that the real roots of different  lengths lie in distinct $W$--orbits. In other words, for any Kac--Moody  Weyl group $W$ there is a symmetrizable Kac--Moody  algebra such that the  lengths of real roots in different orbits are different.

For a symmetric hyperbolic generalized Cartan matrix $A$ of rank $\geq 4$ or symmetrizable hyperbolic generalized Cartan matrix $A$ of rank $\geq 7$, we prove that any 2 real roots of the same length lie in the same $W$--orbit.

We show that if the generalized Cartan matrix $A$ contains zeros then there are simple roots that are stabilized by simple reflections in $W$, that is, $W$ does not act simply transitively on  real roots. We give sufficient conditions in terms of the generalized Cartan matrix $A$ (equivalently ${\mathcal D}$) for $W$ to stabilize a real root.  

We examine some elementary properties of the Weyl group orbits on imaginary roots using symmetry properties of the imaginary light cone in the hyperbolic case. We deduce that the number of $W$--orbits on imaginary roots on a hyperboloid of fixed radius is bounded above by the number of root lattice points on the hyperboloid that intersect the closure of the fundamental region for $W$.

After preparation of this manuscript, we learned that in the finite dimensional case, Carter  has given a characterization of the conjugacy classes of $W$ in the case that $W$ is finite ([Ca]). By identifying roots with $\Z/2\Z$--subgroups Carter's work may be used to deduce some results on the structure of Weyl group orbits when $W$ is finite.

The authors wish to thank Daniel Allcock and Siddhartha Sahi for helpful conversations. \linebreak  We are grateful to Yusra Naqvi for helpful discussions regarding Weyl group orbits on imaginary roots. The authors also wish to thank the referees whose helpful comments led to some improvements in the exposition.

 
 \section{Kac--Moody algebras}

\medskip\noindent  We may construct a Kac--Moody algebra from certain data which  includes a {\it generalized Cartan matrix}. This is a square matrix  $A=(a_{ij})_{i,j\in I}$, $I=\{1,\dots,\ell\}$, whose  entries satisfy the conditions 

\medskip\noindent  $a_{ij}\in{\Z}$, $i,j\in I$,  

\medskip\noindent  $a_{ii}=2$, $i\in I$, 

\medskip\noindent  $a_{ij}\leq 0$ if $i\neq j$, and

\medskip\noindent  $a_{ij}=0\iff a_{ji}=0$.

\medskip
\noindent We may also consider the case that $A$ is {\it symmetrizable:} there exist positive rational numbers $q_1,\dots, q_{\ell}$, such that
the matrix $DA$ is symmetric, where $D=diag(q_1,\dots, q_\ell)$.

\medskip 
\noindent  By a {\it proper submatrix} of $A$, we mean a matrix of the form $A_{\theta}=(a_{ij})_{i,j\in \theta},$ where $\theta$ is a proper subset of $\{1,\dots,\ell\}$.  We say that  $A$ is {\it indecomposable} if there is no partition of the set $\{1,\dots,\ell\}$ into
two non-empty subsets so that $a_{ij}=0$ whenever $i$ belongs to the first subset, while $j$ belongs to the second.

\medskip
\noindent {\bf Possible types}

\medskip
 \noindent {\it Finite type} $A$ is positive-definite, $det(A)>0$. In this case $A$ is the Cartan matrix of a finite dimensional  semisimple Lie algebra.

\medskip
 \noindent{\it Affine type} $A$ is positive-semidefinite, but not positive-definite, $det(A)=0$.

\medskip
 \noindent{\it Indefinite type} $A$ is neither of finite nor affine type, $det(A)<0$.
 
\medskip
 \noindent{\it Hyperbolic type} $A$ is neither of finite nor affine type, but every proper, indecomposable submatrix is either of finite or of affine type, $det(A)<0$.

\medskip
 \noindent If the generalized Cartan matrix $A$ has at least one proper indecomposable matrix of affine type, then we say that $A$ is of {\it noncompact hyperbolic type}.

\newpage
\medskip\noindent{\bf{The Dynkin diagram of a generalized Cartan matrix}}
\medskip

\noindent  The {\it Dynkin diagram} $\mathcal{D}$ of a generalized Cartan matrix $A= (a_{ij})$ is the graph with one 
node for each row (or column) of $A$. If $i\neq j$ and $a_{ij}a_{ji}
\leq 4$ then we connect $i$ and $j$ with  max $( |a_{ij}|, |a_{ji}| )$ edges together with an arrow towards $i$ if $|a_{ij}| > 1$.  
If $a_{ij}a_{ji} > 4$ we draw a bold face line labeled with the ordered pair $(|a_{ij}|$, $|a_{ji}|)$.

\subsection{The Kac--Moody  algebra of a generalized Cartan matrix}

\medskip
\noindent Let $\langle\cdot,\cdot\rangle: \mathfrak{h}\times \mathfrak{h}^{\ast}\to\mathbb{C}$ denote the natural nondegenerate bilinear pairing between a vector space $\mathfrak{h}$ and its dual.

\medskip
 \noindent Given:

\medskip \noindent $\circ$ a generalized Cartan matrix $A=(a_{ij})_{i,j\in I}$, $I=\{1,\dots ,\ell\}$, and

 \noindent $\circ$ a finite dimensional vector space $\mathfrak{h}$ (Cartan subalgebra) with $dim(\mathfrak{h})=2\ell-rank(A)$, and

\noindent $\circ$ a choice of {\it simple roots} $\Pi=\{\alpha_1,\dots,\alpha_{\ell}\}\subseteq \mathfrak{h}^{\ast}$ and {\it simple coroots}
$\Pi^{\vee}=\{\alpha_1^{\vee},\dots,\alpha_{\ell}^{\vee}\} \subseteq
\mathfrak{h}$ such that $\Pi$ and $\Pi^{\vee}$ are linearly independent and such that $\langle\alpha_j,\alpha_i^{\vee}\rangle=\alpha_j(\alpha_i^{\vee})=a_{ij}$, $i,j=1,\dots \ell$,

\medskip
\noindent we may associate a Lie algebra $\mathfrak{g}=\mathfrak{g}(A)$  over $K$, a field, generated by $\mathfrak{h}$ and elements $(e_i)_{i\in I}$, $(f_i)_{i\in I}$ subject to relations (([M1], [M2], and [K], Theorem 9.11):

\medskip
 $[\mathfrak{h}, \mathfrak{h}]=0$,

 $[h,e_i]=\alpha_i(h) e_i$, $h\in \mathfrak{h}$,

 $[h,f_i]=-\alpha_i(h) f_i$, $h\in \mathfrak{h}$,

 $[e_i,f_i]=\alpha_i^{\vee}$,

 $[e_i,f_j]=0,\ i\neq j$,

 $(ad\ e_i)^{-a_{ij}+1}(e_j)=0,\ i\neq j$,

 $(ad\ f_i)^{-a_{ij}+1}(f_j)=0,\ i\neq j$,

\medskip
\noindent where $(ad(x))(y)=[x,y]$.

We have $rank(A)=\ell \iff det(A)\neq 0$. If $det(A)=0$ as in the affine case, then $rank(A)<\ell$.  As usual, $\Delta$ is a subset of $\mathfrak{h}^*$ and we identify $\mathfrak{h}^*$ with $\mathfrak{h}$ via the invariant form $(\cdot\mid\cdot)$ defined in Subsection~\ref{form}.

\subsection{Weyl group}

\medskip
 For each simple root $\alpha_i$, $i=1,\dots, \ell$  we define the simple root reflection
$$w_i(\alpha_j)
    =
    \alpha_j - \alpha_j(\alpha_i^{\vee})\alpha_i.$$
It follows from the formula that $w_i(\alpha_i)=-\alpha_i$.  The $w_i$ generate a subgroup $$W=W(A)\subseteq Aut(\mathfrak{h}^{\ast}),$$ called the {\it Weyl group} of $A$. The group $W$ acts on the set of all roots. For $i,j\in I$, and for $i\neq j$, we set
$$c_{ii}=1,\quad c_{ij}=2,3,4,6,{\text{ or }}\infty$$
according as
$$a_{ij}a_{ji}=0,1,2,3,{\text{ or }}\geq 4$$
respectively. Then $W=W(A)$ is the Coxeter group with presentation: 
$$W=\langle w_i\mid i\in I, (w_iw_j)^{c_{ij}}=1,\ if\ c_{ij}\neq\infty\rangle.$$
If $a_{ij}a_{ji}\geq 4$, then $c_{ij}=\infty$ however the relation $`(w_iw_j)^{\infty}=1'$, meaning that the element $w_iw_j$ has infinite order, is not explicitly included in the presentation. 

 If $\Mg$ is of finite type, $W$ is a finite group. If  $\Mg$ is infinite dimensional, $W$ is infinite ([K], Ch 3).  The group $W$ acts on the set of all roots.

\subsection{Invariant form}\label{form} If $A$ is a symmetrizable generalized Cartan matrix,  the algebra $\mathfrak{g}=\mathfrak{g}(A)$ admits a well--defined non--degenerate symmetric bilinear form $(\cdot\mid\cdot)$ which plays the role of `squared length' of a root ([K], Theorem 2.2). If $\mathfrak{g}$ is of finite type, then $(\cdot\mid\cdot)$ is the usual Killing form. The form $(\cdot\mid\cdot)$  is an analog of the Killing form if $\mathfrak{g}=\mathfrak{g}(A)$ is infinite dimensional and $(\cdot\mid\cdot)\mid_{\mathfrak{h}^{\ast}}$ is $W$--invariant.

\subsection{Real and imaginary roots} 

\medskip\noindent If $\mathfrak{g}=\mathfrak{g}(A)$ is  finite dimensional, all roots are Weyl group translates of simple roots. That is,
$$\Delta\quad=\quad W\Pi.$$
For infinite dimensional Kac--Moody root systems, there are additional mysterious roots of negative norm (`squared length') called imaginary roots.  

\medskip\noindent  A root $\alpha\in\Delta$ is called a {\it real root} if there exists $w\in W$ such that $w\alpha$ is a
simple root. A root $\alpha$ which is not real is called {\it
imaginary}. We denote by $\Phi=\Delta^{re}$ the real roots,  $\Delta^{im}$ the
imaginary roots, $\Phi_{\pm}=\Delta^{re}_{\pm}$ the positive and negative Weyl roots,
and $\Delta^{im}_{\pm}$ the positive and negative imaginary roots. Then
$$\Delta\quad=\quad\Delta^{re}\sqcup\Delta^{im}$$
$$\Delta^{re}\quad=\quad \Delta^{re}_+\sqcup\Delta^{re}_-,$$
$$\Delta^{im}\quad=\quad\Delta^{im}_+\sqcup\Delta^{im}_-,$$
and
$$\Phi\quad=\quad W\Delta^{re}\quad = \quad W\Pi.$$
Real and imaginary roots may be characterized by their `squared length' $(\alpha\mid\alpha)$ for $\alpha\in\Delta$. If $A$ is symmetrizable, we have ([K], Prop 5.1 and 5.2):
$$\alpha\in\Delta^{re}\iff (\alpha\mid\alpha)>0$$
$$\alpha\in\Delta^{im}\iff (\alpha\mid\alpha)\leq 0.$$

\subsection{Reflection geometry of roots}

 Let $\mathfrak{h}$ be a Cartan subalgebra of a Kac--Moody  algebra $\mathfrak{g}$ and let $\alpha$ be a root. Associated to  $\alpha$ is a root vector $h_{\alpha}\in \mathfrak{h}$, a hyperplane $H_{\alpha}$ and an involution, $s_{\alpha}$ such that $s_{\alpha}|_{H_{\alpha}}=\Id$ and $s_{\alpha}(h_{\alpha})=-h_{\alpha}$. Hence, with respect to the invariant form $(\cdot\mid\cdot)$, the hyperplane $H_{\alpha}$ is orthogonal to the root vector $h_{\alpha}$.

Let now $\beta$ denote another root. The following observation is clear:

$$s_{\alpha}(\beta)=\beta\qquad \textrm{if and only if}\qquad h_{\beta}\in H_{\alpha}\,.$$

Assume $\alpha$ and $\beta$ are simple roots in a basis $\Pi$ of the root system. Then $$\qquad h_{\beta}\in H_{\alpha}\qquad \textrm{if and only if}\qquad  a_{\alpha,\beta}=a_{\beta,\alpha}=0\,$$
where $a_{\alpha,\beta}=a_{\beta,\alpha}$ is the $\alpha,\beta$--entry of the generalized Cartan matrix.

If $\alpha$ and $\beta$ generate a finite root system of type $B_2$ or $C_2$ such that $\alpha$ is a short root and $\beta$ is a long root, then the roots $\pm(\alpha+\beta)$ are in $H_{\alpha}$ and the roots $\pm (\beta+2\alpha)$ are in $H_{\beta}$.   

The following lemma is a rewording of [K], Lemma (3.7).

\begin{lemma}\label{posroot} Let $\Mg$ be a symmetrizable Lie algebra or Kac--Moody  algebra. Let $\alpha$ be any positive root. Let $w_j$ be a simple root reflection. Then $w_j\alpha$ is also a positive root, unless $\alpha$ equals the simple root $\alpha_j$ in which case $w_j\alpha_j=-\alpha_j$.
\end{lemma}

We will use these observations in the sections that follow.

\section{Weyl group orbits on real roots}

\noindent Let $A$ be a  generalized Cartan matrix. Let $\mathcal{D}=\mathcal{D}(A)$ be the corresponding Dynkin diagram with vertices indexed by $I=\{1,2,\ldots,\ell\}$.  We let $\mathcal{D}_{\ast}$ denote the graph obtained from $\mathcal{D}$ by deleting all multiple edges, including arrows and edge labels. Let $\mathcal{D}_1,\dots,\mathcal{D}_s$ denote the connected subdiagrams of $\mathcal{D}_{\ast}$. We call $\mathcal{D}_{\ast}$ the {\it skeleton} of $\mathcal{D}$. We may describe the graph $\mathcal{D}_{\ast}$ as follows

$Vertices(\mathcal{D}_{\ast})=Vertices(\mathcal{D})$ with the same labelling, that is, indexed by $I=\{1,2,\ldots,\ell\}$

$Edges(\mathcal{D}_{\ast})=\bigcup_{i=1}^s \ Edges(\mathcal{D}_i)$

\noindent  Vertices $i$ and $j$ are adjacent in $\mathcal{D}_{\ast}$ if and only if $i$ and $j$  are connected in $\mathcal{D}$ by a single edge with no arrows or edge labels. This occurs if and only if  $a_{ij}=a_{ji}=-1$. 

In the discussions that follow, we let $A$ be a generalized Cartan matrix, $W(A)$ be its Weyl group, $\Pi$ the set of simple roots indexed by $I$,  $\Phi$ the set of real roots, and $\mathcal{D}$ the corresponding Dynkin diagram.  Our main objective is to classify the Weyl orbits on $\Phi$ using only the associated Dynkin diagram.  This can be accomplished using a lemma of Tits in association with the theorem below, which requires some results regarding symmetric rank 2 root subsystems.

If $\Pi=\{\alpha_1,\ldots,\alpha_{\ell}\}$, then any real root $\beta\in\Phi$ can be written as a linear combination of the simple roots,
$$\beta=\sum^{\ell}_{i=1} b_i \alpha_i,$$
where the coordinates $\{b_i\}$ are integers, and are uniquely determined by $\beta$.  Moreover, the coordinates are either all positive or all negative.  In the results that follow, we will need to focus our attention on specific coordinates of a given real root.

Let $\Pi=\{\alpha_1,\ldots,\alpha_{\ell}\}$. Let $\Pi^{\vee}=\{\rho_1,\dots,\rho_{\ell}\}$ be  the dual basis to $\alpha_i$. Combinatorially, for a real root $\beta$,  with $\beta=\sum^{\ell}_{i=1} b_i \alpha_i$, $\rho_i(\beta)$ is the $i^{th}$ coordinate of $\beta$ with respect to the simple roots $\Pi$.
For simple root $\alpha_j$, $\rho_i(\alpha_j)=\delta_{ij}$, where $\delta_{ij}$ is the Kronecker delta function.  We will use these combinatorial details in the results that follow.

Now consider the rank 2 symmetric generalized Cartan matrix $A$,

$$A = \left(\begin{array}{cc}2 & -a \\-a & 2 \end{array}\right),\ a\geq 2.$$   
 We recall that $W=W(A)$ is the infinite dihedral group 
 $$D_{\infty}=\Z\rtimes \Z/2\Z= \langle a,x \mid x^2 = 1,\  x^{-1}ax = a^{-1} \rangle.$$

The following lemma follows easily from the fact that the transformation $(w_1w_2)$ in $W$ has infinite order and generates the $\mathbb{Z}$--factor of $D_{\infty}$. 
\begin{lemma}\label{lem1}
Let $A = \left(\begin{array}{cc}2 & -a \\-a & 2 \end{array}\right)$ for $a \ge2$, $\Pi=\{\alpha_1, \alpha_2\}$ simple roots.  Let $\beta_n=(w_2w_1)^n(-\alpha_1)$ and $\gamma_n=(w_1w_2)^n(-\alpha_1)$ for some $n\in\mathbb{Z}_{\ge0}$.  Then
\begin{enumerate}
	\item  $(\rho_1(\beta_n))_{n\in\mathbb{Z}_{\ge0}}$  is an increasing sequence of integers.
	\item $(\rho_1(\gamma_n))_{n\in\mathbb{Z}_{\ge0}}$ is a strictly decreasing sequence of integers.
\end{enumerate}
\end{lemma}

We are now ready to show that the Weyl group $W$ is not transitive on the symmetric rank 2 root system investigated above if $a\ge 2$.

\noindent
\begin{lemma}\label{lem2}
Let $A$ be the rank 2 generalized Cartan matrix $A = \left(\begin{array}{cc}2 & -a \\-a & 2 \end{array}\right)$, and $\mathcal{D}$ its associated Dynkin diagram with 2 vertices. Let $\Pi=\{\alpha_1, \alpha_2\}$ be the simple roots.  If $a \ge2$ (equivalently if $\mathcal{D}$ is affine or hyperbolic), then $W\{\alpha_1\}\cap W\{\alpha_2\}=\emptyset$, that is, there is no $w\in W$ such that $w(\alpha_1)=\alpha_2$.
\end{lemma}
\begin{proof}
Assume there is an element $w\in W$ such that $w(\alpha_1)=\alpha_2$.  Without loss of generality, assume $w$ is reduced (or simply replace $w$ with its corresponding unique reduced word).   Since $W$ is generated by 2 simple root reflections, the first `letter' of $w$ must either be $w_1$ or $w_2$.  

 \textit{Case 1}: Assume $w$ begins with $w_1$, so $w=u\cdot w_1$ for some Weyl group element $u\in W$. Then
$$\alpha_2=w(\alpha_1)=u\cdot w_1(\alpha_1)=u(-\alpha_1).$$

Both $-\alpha_1$ and $\alpha_2$ lie on the upper branch of $\Phi$, therefore $u$ must be a translation of the form 
$u=(w_1w_2)^m$ or $u=(w_2w_1)^m$ for some $m\in\mathbb{Z}_{>0}$.  However, since $w$ is reduced, $u$ must begin with $w_2$, so we omit the latter case.  Thus,
$$u(-\alpha_1)=(w_1w_2)^m(\alpha_1)=\alpha_2.$$

By Lemma ~\ref{lem1}$(ii)$, we now have the following decreasing sequence of integers:
$$\rho_1(-\alpha_1)>\rho_1\left(w_1w_2(-\alpha_1)\right)> \cdots \ >\rho_1\left((w_1w_2)^m(-\alpha_1)\right)>\cdots.$$

For $n=0$, $\rho_1(-\alpha_1)=-1$, and for $n=m$, $\rho_1((w_1w_2)^m(-\alpha_1))=\rho_1(\alpha_2)=0$. So the sequence is
$$-1>\cdots>0>\cdots,$$

which is false.  Therefore $w$ cannot begin with $w_1$.  

 \textit{Case 2}: Now assume $w$ begins with $w_2$, so $w=v\cdot w_2$ for some Weyl group element $v\in W$. Then
$$\alpha_2=w(\alpha_1)=v\cdot w_2(\alpha_1)=v(\alpha_1+a\alpha_2).$$

Both $\alpha_2$ and $\alpha_1+a\alpha_2$ lie on the upper branch of $\Phi$, and since $w$ is reduced, $v$ must begin with $w_1$. Therefore, $v$ is a translation of the form $v=(w_2w_1)^m$ for some $m\in \mathbb{Z}$.  Thus,
$$v(\alpha_1+a\alpha_2)=(w_2w_1)^m(\alpha_1+a\alpha_2)=\alpha_2.$$

 Observe that $\alpha_1+a\alpha_2=(w_2w_1)(-\alpha_1)$.  Therefore, by Lemma ~\ref{lem1}$(i)$, we have the following increasing sequence of integers: 
$$\rho_1(\alpha_1+a\alpha_2)<\rho_1\left(w_2w_1(\alpha_1+a\alpha_2)\right)< \cdots \ <\rho_1\left((w_2w_1)^m(\alpha_1+a\alpha_2)\right)<\cdots.$$

But $\rho_1(\alpha_1+a\alpha_2)=1$ and $\rho_1((w_2w_1)^m(\alpha_1+a\alpha_2))=\rho_1(\alpha_2)=0$, so the sequence is
$$1<\cdots<0<\cdots,$$

which is also false.  Thus, $w$ does not begin with $w_2$.

\end{proof}

\begin{lemma}\label{lem3} Let $A$ be any generalized Cartan matrix, $\beta \in \Phi$, $r\in I$.  Then if $i\not=r$,
$$\rho_i(w_r(\beta))=\rho_i(\beta).$$
 
\end{lemma}
\begin{proof} Let $w_r$ act on a root $\beta\in\Phi$:
$$w_r\beta=w_r\sum_{i=1}^{\ell} b_i\alpha_i = \sum_{i=1}^{\ell} b_i(w_r\alpha_i)= \sum_{i=1}^{\ell} b_i(\alpha_i-a_{ir}\alpha_r)= \sum_{i=1 \ i \neq r}^{\ell} b_i \alpha_i + \left(b_r-\sum_{i=1}^{\ell} b_i a_{ir}\right)\alpha_r.$$
 Thus we see that for all $i\not=r$, $\rho_i(w_r(\beta))=b_i=\rho_i(\beta)$.

\end{proof}

Our objective is to prove that two simple roots lie in the same $W$--orbit if and only if their corresponding vertices lie in the same connected component of $D_*$. In other words, their corresponding vertices in the Dynkin diagram are connected by a path of single edges.  To prove the forward argument we use the following lemma of Tits, which we will strengthen with the aid of the foregoing results.

\noindent
\textbf{Lemma 13.31 (\cite{T})}. \textit{Let $\mathcal{D}$ be a Dynkin diagram, $W$ its associated Weyl group, and $\Pi$ the set of simple roots indexed over a set $I$.  If for some $j,k\in I$, $j\not=k$  there exists $w\in W$ such that 
$$w\alpha_j=\alpha_k,$$
then there exists a sequence $\{u_1,\ldots,u_{m-1}\}$ of elements of $W$,  and there exist sequences \newline $\{j=i_1,\ldots,i_m=k\}$, $i_t\in I$, and $\{s_1,\ldots,s_{m-1}\}$, $s_t\in I$ such that
$$w=u_{m-1}u_{m-2}\cdots u_2u_1$$
where $u_t(\alpha_{i_t})=\alpha_{i_{t+1}}$ and $u_t\in \langle w_{i_t}, w_{s_t} \rangle$, $t\in\{1,\ldots,m-1\}.$}

Essentially, if two simple roots $\alpha_j$ and $\alpha_k$ lie in the same Weyl orbit (that is, they are connected by a reduced Weyl group element $w$), then Tits' lemma concludes that $w$ is a product of elements from the first sequence $\{u_t\}$.  Moreover, we observe that the second sequence $\{i_t\}$ corresponds to a sequence of simple roots, starting with $\alpha_j$ and ending with $\alpha_k$.  Since each $u_t$ maps $\alpha_{i_t}$ to the next simple root $\alpha_{i_{t+1}}$ in this sequence, the Weyl group element $w$ maps $\alpha_j$ to $\alpha_k$ through the other $m-2$ simple roots in sequential order.  Figure ~\ref{fig:tits} illustrates this conclusion.

\begin{figure}[h!]
\includegraphics[width=2.5 in]{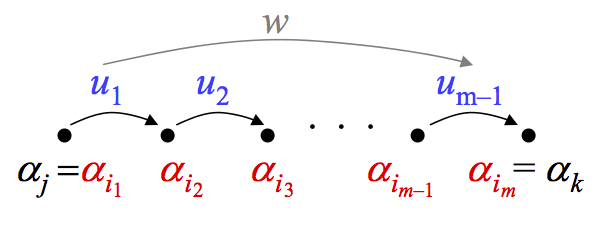}
\caption{The conclusion of Tits' lemma.}
\label{fig:tits}
\end{figure}

We must prove that the  hypotheses of the Lemma are sufficient for vertices $j=i_1$ and $k=i_m$ in the Dynkin diagram to be connected by a path of single edges.  If the sequence $\{j=i_1,\ldots,i_m=k\}$ corresponds to such a path, then we will need a full description of the rank 2 root subsystems generated by each pair of simple roots $\alpha_{i_t}$ and $\alpha_{i_{t+1}}$ in the sequence.  

However, the statement of the Lemma does not clearly define the  elements $u_t$ which comprise $w$, that is, for each $u_t$ a full description of the simple reflection $w_{s_t}$ is not provided.  In addition, if $u_t$ is a reduced word in $\langle w_{i_t}, w_{s_t} \rangle$, then there is no indication of its length.

In the next theorem, we determine these unknowns without ambiguity, and we follow with a much sharper statement of the lemma of Tits.  This later result will give us the sufficiency condition needed to allow us to classify the Weyl group orbits of any hyperbolic Kac--Moody  root system based solely on its associated Dynkin diagram.

\begin{theorem}\label{thm1}
Let $\mathcal{D}$ be a Dynkin diagram, $\Pi$ a basis of its associated root system, $I$ the index set for $\Pi$, and $W$ the corresponding Weyl group.  We assume that $W$ is generated by $\geq 2$ simple root reflections. Let $\widetilde{W}\leq  W$ be a subgroup generated by exactly 2 simple Weyl group reflections. For $j,k\in I$, let $A_{(j,k)}$ be the rank 2 submatrix of $A$,
\[
A_{(j,k)} =
\left(
\begin{array}{cc}
2 & a_{jk} \\
a_{kj} & 2 \end{array}
\right).
\]
If there exists a reduced word $w\in \widetilde{W}$ such that $w(\alpha_j)=\alpha_k$, then

(1) $A_{(j,k)}=A_2$,

(2) Vertices $j,k$ are adjacent in $\mathcal{D}_*$.

(3) $w=w_jw_k$,

(4) $\widetilde{W}= \langle w_j, w_k\rangle$,

\end{theorem}

\begin{proof}

(1)  Since $W$ preserves root length, and $\alpha_k$ is a $W$--translate of $\alpha_j$, $A_{(j,k)}$ must be symmetric, else $\alpha_j$, $\alpha_k$ are of different lengths.  Thus we have
\[
A_{(j,k)} =
\left(
\begin{array}{cc}
2 & -a \\
-a & 2 \end{array}
\right),
\]

 where $a=-a_{jk}=-a_{kj}\in\mathbb{Z}_{\ge 0}$, and $a\ge 0$.  Assume $a\not=1$.

 \textit{Case 1}: If $a=0$ then $A_{(j,k)}$ is the Cartan matrix associated to $A_1 \times A_1$. Calculating the $\widetilde{W}-$orbit of $\alpha_j$ we have
$$w_k\alpha_j=\alpha_j-(\langle\alpha_j,\alpha_k^\vee\rangle)\alpha_k=\alpha_j-(a)\alpha_k=\alpha_j,$$
$$w_j\alpha_j=-\alpha_j,$$
$$w_k(-\alpha_j)=-w_k(\alpha_j)=-\alpha_j,$$
$$w_j(-\alpha_j)=\alpha_j,$$
and thus, $\widetilde{W}\{\alpha_j\}=\{\pm\alpha_j\}$, which does not contain $\alpha_k$, a contradiction.

\textit{Case 2}: If $a\ge 2$, then $A_{(j,k)}$ is of affine or hyperbolic type, so by Lemma ~\ref{lem2}, $W\{\alpha_j\}\cap W\{\alpha_k\}=\emptyset$, contradicting the existence of $w\in\widetilde{W}$ where $w(\alpha_j)=\alpha_k$.

 Thus, $a$ must equal 1, and so $A_{(j,k)} = A_2$.

(2)  By part (1), $\alpha_j$ and $\alpha_k$ span an $A_2$ root subsystem.  Thus vertices $j$ and $k$ are adjacent in $\mathcal{D}_*$ and thus in $\mathcal{D}$ (Figure~\ref{adj}).

\begin{figure}[h!]
\includegraphics[width=1 in]{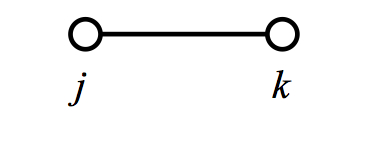}
  \caption{ Adjacent vertices $j$ and $k$  in $\mathcal{D}_*$ }
  \label{adj}
\end{figure}

(3) Recalling from Lemma ~\ref{lem1} that $w_i\alpha_j=\alpha_j+a\alpha_i$ and using $a=1$, we have
\begin{align*}
w_jw_k\alpha_j&=
w_j(\alpha_j+\alpha_k)\\  
&=\alpha_j+\alpha_k-\langle\alpha_j+\alpha_k,\alpha_j^\vee\rangle\alpha_j\\  
&=\alpha_j+\alpha_k-\left(\langle\alpha_j,\alpha_j^\vee\rangle+\langle\alpha_k,\alpha_j^\vee\rangle\right)\alpha_j\\  
&=\alpha_j+\alpha_k-(a_{jj}+a_{kj})\alpha_j\\
&=\alpha_j+\alpha_k-(2-1)\alpha_j\\ 
&=\alpha_k .
\end{align*}

(4) By part (3) we see that $w_j, w_k \in \widetilde{W}.$  By assumption $\widetilde{W}$ is generated by two simple root reflections, thus $\widetilde{W}= \langle w_j, w_k\rangle$.

\end{proof}

We now obtain a strengthening of Tits' lemma as a corollary.

\begin{corollary}\label{newtits}
Let $\mathcal{D}$ be a Dynkin diagram, $W$ its associated Weyl group, and $\Pi$ the set of simple roots indexed over a set $I$.  Let $w\in W$ and $j, k\in I$  be such that $w\alpha_j=\alpha_k$.  Then there exists a sequence $\{j=i_1, i_2, \ldots, i_m=k\}\subset I$ such that 
$$w=(w_{i_{m-1}}w_k)(w_{i_{m-2}}w_{i_{m-1}})(w_{i_{m-3}}w_{i_{m-2}})\cdots(w_{i_2}w_{i_3})(w_jw_{i_2}),$$
and $i_t, i_{t+1}$ are adjacent in $\mathcal{D}_{\ast}$ for all $t\in\{1,\ldots m-1\}$. Thus, $j$ and $k$ are connected by a path of single edges in $\mathcal{D}$.
\end{corollary}
\begin{proof} By the lemma of Tits, we have a sequence of Weyl group elements $u_1, u_2, \ldots,u_{m-1}$ and a sequence $\{s_1, \ldots, s_{m-1}\}$ such that
$$w=u_{m-1}\cdots u_2u_1$$
where 
\begin{enumerate}
	\item  $u_t(\alpha_{i_t})=\alpha_{i_{t+1}}$ and
	\item  $u_t \in \langle w_{i_t}, w_{s_t}\rangle.$ 
\end{enumerate}
The description of $w_{s_t}$ here is made clear by applying Theorem~\ref{thm1}. The last two conditions together imply that $s_t=i_{t+1}$. Moreover, they imply that for all $t\in\{1,\ldots m-1\}$, $i_t$ and $i_{t+1}$ are adjacent, and $u_t=(w_{i_t} w_{i_{t+1}})$.  Thus we may rewrite the Weyl group element $w$ as 
$$w=(w_{i_{m-1}}w_k)(w_{i_{m-2}}w_{i_{m-1}})(w_{i_{m-3}}w_{i_{m-2}})\cdots(w_{i_2}w_{i_3})(w_jw_{i_2}),$$
and so $j$ and $k$ are connected by a path $\{j=i_1, i_2, \ldots, i_m=k\}$ in $\mathcal{D}_{\ast}$. Thus they are in the same connected component. 
\end{proof}

 Let  $\mathcal{D}^J_{\ast}$ denote a maximal connected subdiagram (connected component) of $\mathcal{D}_{\ast}$ whose vertices are indexed by $J\subseteq I$. In strengthening the lemma of Tits, we now make clear that the sequence $\{i_1,\ldots, i_m\}$ indeed corresponds to a sequence of vertices in $D_*^J$ for some $J$, where each pair $i_t,i_{t+1}$ is connected by a single edge.  Thus the condition that $\alpha_j$ and $\alpha_k$ lie in the same Weyl group orbit is sufficient for vertices $j, k \in \mathcal{D}$ to be connected by a path of single edges.  Now we show that this condition is also necessary.

Let $\mathcal{D}$ be a Dynkin diagram, $\mathcal{D}_*$ its skeleton, $\Pi=\{\alpha_1,\alpha_2,\ldots,\alpha_{\ell}\}$ a choice of simple roots and $W$ the Weyl group.

\begin{lemma} \label{chain} Suppose vertices $i$ and $j$ are adjacent in $\mathcal{D}_{\ast}$. Then $\alpha_i=w_jw_i\alpha_j$, where $w_j$ is the simple root reflection corresponding to $\alpha_j$. \end{lemma}

\begin{proof}Since $i$ and $j$ are adjacent in $\mathcal{D}_{\ast}$, there is a single edge between $i$ and $j$ in $\mathcal{D}$, and hence\newline  $a_{ij}=a_{ji}=-1$ in the generalized Cartan matrix $A$.  The proof of Theorem ~\ref{thm1} part$(4)$ has already established that if $a_{ij}=a_{ji}=-1$, then $w_jw_i$ maps $\alpha_j$ to $\alpha_i$.
\end{proof}

 The following corollary is immediate.

\begin{corollary} \label{chaincor} If $i_1, i_2,\dots,i_m$ are the vertices of a path in $\mathcal{D}_{\ast}$, then
\begin{enumerate}
\item  There exists $w\in \langle w_{i_1},\dots , w_{i_m}\rangle$ such that  $\alpha_{i_m}=w\alpha_{i_1}$.
\item We have $a_{i_1i_2}\ =\ a_{i_2i_3}\ =\ \dots\ =\ a_{i_{m-1}i_m}\ =\ a_{i_2i_1}\ =\ a_{i_3i_2}\ =\ \dots\ =\ a_{i_mi_{m-1}}\ =\ -1.$
\end{enumerate}
\end{corollary}

\begin{theorem} \label{necessary} Let $\mathcal{D}$ be a Dynkin diagram, let $\Pi=\{\alpha_1,\dots ,\alpha_{\ell}\}$ be the simple roots of the corresponding root system and let $W$ denote the Weyl group. Let  $\mathcal{D}^J_{\ast}$ denote a maximal connected subdiagram  of $\mathcal{D}_{\ast}$ whose vertices are indexed by $J\subseteq I$. Then

 (1) For all $i,j\in J$ there exists $w\in {W}$ such that $\alpha_i=w\alpha_j$, thus ${W}\alpha_i={W}\alpha_j$.

 (2) If $j\in J$ and $k \notin  J$, we have ${W}\alpha_j\cap {W}\alpha_k=\varnothing$.

\end{theorem}

\begin{proof} (1) Since $i,j\in J$ , there exists a connected path of vertices $\{i=i_1, i_2, \ldots, i_{m-1}, i_m=j\}$ in $\mathcal{D}^J_{\ast}$ and a corresponding subgroup $\langle w_{i_1},\dots , w_{i_m}\rangle \subset W$.  Thus we may use Corollary ~\ref{chaincor}$(i)$, and the result follows.

 (2) Assume $j\in J$ and $k\notin J$, so $j,k$ are not in the same connected component.  Assume that $W\{\alpha_j\}=W\{\alpha_k\}$.  Then there exists a $w\in W$ such that $w\alpha_j=\alpha_k$.  By Corollary ~\ref{newtits}, there exists a sequence $\{j=i_1, i_2, \ldots, i_m=k\}\subset I$ such that 
$$w=(w_{i_{m-1}}w_k)(w_{i_{m-2}}w_{i_{m-1}})(w_{i_{m-3}}w_{i_{m-2}})\cdots(w_{i_2}w_{i_3})(w_jw_{i_2}),$$
and $i_t, i_{t+1}$ are adjacent in $\mathcal{D}_{\ast}$ for all $t\in\{1,\ldots m-1\}$.  Thus there exists a path of vertices $\{j=i_1, i_2, \ldots, i_m=k\}$ connecting $j$ and $k$.  But then $k\in J$, which is a contradiction.  Therefore ${W}\alpha_j\cap {W}\alpha_k=\varnothing$.
\end{proof}

 With Theorem ~\ref{necessary} and Corollary ~\ref{newtits}, we now have the necessary and sufficient conditions for two distinct simple roots to be in the same Weyl group orbit.  This result is summarized in the following corollary, which is immediate.
      
\begin{corollary} \label{maincor}  Let $\mathcal{D}$ be a Dynkin diagram, let $\Pi=\{\alpha_1,\dots ,\alpha_{\ell}\}$ be the simple roots of the corresponding root system and let $W$ denote the Weyl group. Then for $i\neq j$, the simple roots $\alpha_i$ and $\alpha_j$ are in the same $W$--orbit if and only if vertices $i$ and $j$  in the Dynkin diagram corresponding to $\alpha_i$ and $\alpha_j$ are connected by a path consisting only of single edges.
\end{corollary}

Moreover, we have the following stronger statement.

\begin{corollary} \label{realroots} Let $\mathcal{D}$ be a Dynkin diagram, let $\Pi=\{\alpha_1,\dots ,\alpha_{\ell}\}$ be the simple roots of the corresponding root system and let $W$ denote the Weyl group. Let  $\mathcal{D}^{J_1}_{\ast}$, $\mathcal{D}^{J_2}_{\ast}$, $\dots$,  $\mathcal{D}^{J_t}_{\ast}$ denote the connected subdiagrams of $\mathcal{D}_{\ast}$ whose vertices are indexed by $J_1,J_2,\dots ,J_t\subseteq I$. Then the set of real roots $\Phi=W\Pi$ is
$$\Phi=W\{\alpha_{J_1}\}\amalg W\{\alpha_{J_2}\}\amalg\dots \amalg W\{\alpha_{J_t}\},$$
where $\alpha_{J_s}$ denotes the subset of simple roots indexed by the subset $J_s \subseteq I$.
\end{corollary}

 Given any Dynkin diagram $\mathcal{D}$, we can therefore determine the disjoint orbits of the Weyl group on real roots by determining the skeleton $\mathcal{D}_{\ast}$. Applying the classification of hyperbolic Dynkin diagrams (\cite{CCCMNNP}, \cite{KM}, \cite{Li}, \cite{Sa}), it follows that maximal number of disjoint Weyl group orbits in  a hyperbolic root system is 4.

\section{A Cayley graph encoding the Weyl group action on real roots}

Let $A$ be an $\ell\times\ell$ generalized Cartan matrix with Kac--Moody  algebra $\mathfrak{g}=\mathfrak{g}(A)$, root system $\Delta=\Delta(A)$, real roots $\Phi\subseteq\Delta$, and Weyl group $W=W(A)$. Let $\Pi=\{\alpha_1,\ldots,\alpha_{\ell}\}$ denote a fixed base of $\Delta$. Using only the action of the simple root reflections on $\Phi$, we define a graph $\mathcal{P}$ associated to $\Phi$, relative to $\Pi$, which has the following properties:
\begin{enumerate}
\item The vertices of $\mathcal{P}$ are in one-to-one correspondences with the elements of $\Phi$.
\item The number of connected components of $\mathcal{P}$ equals the number of disjoint orbits of $W$ on $\Phi$, and $W$ acts transitively on  each connected component of $\mathcal{P}$.
\item The vertices of $\mathcal{P}$ in a given connected component are in 1--to--1 correspondence with the roots in the corresponding Weyl orbit.
\item The number of orbits of $W$ on $\Phi$ (equivalently, the number of connected components of $\mathcal{P}$) can be determined from the generalized Cartan matrix $A$ (equivalently from the Dynkin diagram $\mathcal{D}$ of $A$).
\end{enumerate}

\medskip\noindent
The graph ${\mathcal{P}}$ is defined as follows.

Let $A$ be a (generalized) Cartan matrix, $A=(A_{ij})_{i,j=1,2,...,l}$.
Let $\Pi$ be the root basis for $A$, ${W}$ be the Weyl group corresponding to $A$ and $\Phi={W}\Pi$ be the real roots of $A$. We define a labelled graph $(\mathcal{P},\nu)$ which encodes the action of $W$ on $\Phi$ where $\mathcal{P}=(V\mathcal{P},E\mathcal{P})$ is an unoriented graph (with vertices $V\mathcal{P}$ and edges $E\mathcal{P}$) and $\nu:E\mathcal{P}\rightarrow S$ is a labelling of a simple reflection to each edge of $\mathcal{P}$.  Vertices and edges of $\mathcal{P}$ are defined as follows:
\begin{align*}
&\textrm{Vertices}(\mathcal{P})=\Phi\qquad\text{(the set of real roots)}\\
&\textrm{Edges}(\mathcal{P})\subseteq\Phi\times\Phi
\end{align*}
where $\alpha$ and $\beta$ are adjacent in $\mathcal{P}$ if and only if
\[w_i\alpha=\beta\quad\textrm{for some simple reflection }w_i\textrm{ and}\]
\[\nu(e)=w_i\quad\textrm{where }e=(\alpha,\beta)\textrm{ is an edge of }\mathcal{P}.\]
\noindent We will customarily label the vertices with the corresponding real root expressed either as a linear combination of simple roots.

\medskip\noindent
It should be noted that, since every $w_i\in W$ is an involution, every edge in $\mathcal{P}$ represents a bigon (Figure~\ref{bigon}).

\begin{figure}[h!]
\includegraphics[width=.8 in]{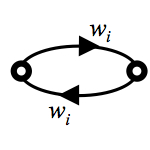}
\caption{The bigon $w_i^2=1$}
\label{bigon}
\end{figure} 

By definition of  $\mathcal{P}$, each orbit of $W$, respectively each connected component of  $\mathcal{P}$,  contains at least one simple root. Moreover  $\mathcal{P}$  is connected if and only if $W$ acts transitively on  $\Phi$, and $\mathcal{P}$ has no loops if and only if $W$ acts simply transitively on  connected components, that is, with no fixed roots.

Let $\mathcal{P}_j$, ($j=1,\ldots,s$) denote the connected components of $\mathcal{P}$. In the case that $W$ acts simply transitively on  the vertices of ${\mathcal {P}}$, there is a natural isomorphism from ${\mathcal {P}}$ to the Cayley graph of $W$. In fact, if $W$ acts simply transitively on  a connected component ${\mathcal {P}}_j$, then there is a natural isomorphism from ${\mathcal {P}}_j$ to the Cayley graph of $W$. There are important examples of finite, affine and hyperbolic root systems for which this criterion holds.

By (ii) the action of $W$ is transitive on each connected component ${\mathcal {P}}_j$. If, however, the action of $W$ is not simply transitive on some $\mathcal{P}_j$, then there is a bijective correspondence between the vertices of ${\mathcal {P}}_j$ and elements of the coset space $W/W_{\alpha_{i_j}}$, where $\alpha_{i_j}$ is a choice of simple root from the connected component ${\mathcal {P}}_j$ and  $W_{\alpha_{i_j}}$ denotes the subgroup of $W$ stabilizing $\alpha_{i_j}$, well defined up to conjugacy. This correspondence induces an isomorphism of ${\mathcal {P}}_j$ with a certain `orbifold' associated to the Cayley graph of $W$.

We may therefore consider the graph ${\mathcal {P}}$ as a generalization of the Cayley graph of $W$, namely we may view $\mathcal{P}$ as `the Cayley graph of the action of $W$ on $\Phi$ relative to $\Pi$.' This notion may be generalized further to associate a graph $\mathcal{P}$ to any action of a group $G$ on a set $X$ relative to a choice of generating set for $G$.



\section{Some Finite, Affine and Hyperbolic Examples} \label{sec4}
Consider the Cartan matrix $A_1 = (2)$ for the $A_1$ root system. Then $\Phi=\{\alpha,-\alpha\}$, and $\mathcal{P}(A_1)$ is given in Figure~\ref{A1orbit}.

\begin{figure}[h!]
\includegraphics[width=1.4 in]{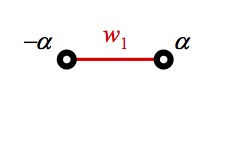}
\caption{Orbit of type $A_1$}
\label{A1orbit}
\end{figure}

For \[A_2=\left( \begin{array}{ccc} 2 & -1\\ -1 & 2\end{array}\right),\]  we have
$\Phi=\{\pm\alpha_1,\pm\alpha_2,\pm(\alpha_1+\alpha_2)\}$, and $\mathcal{P}(A_2)$ is given in Figure~\ref{A2orbit}.

\begin{figure}[h!]
\includegraphics[width=2.5 in]{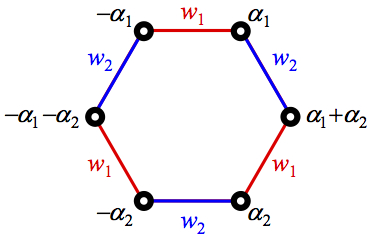}
\caption{Orbit of type $A_2$}
\label{A2orbit}
\end{figure}

$\mathcal{P}(A_2)$ illustrates that the Weyl group acts transitively on  all roots in $A_2$: $\Phi(A_2)=W\alpha_1=W\alpha_2$. 

An example in which $\mathcal{P}$ is disconnected is 
\[B_2=\left( \begin{array}{ccc} 2 & -2\\ -1 & 2\end{array}\right),\]  whose simple roots have two distinct lengths.  The Dynkin diagram is given in Figure~\ref{B2}.

\begin{figure}[h!]
\includegraphics[width=1 in]{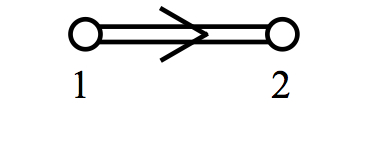}
\caption{Dynkin diagram of type $B_2$}
\label{B2}
\end{figure}
has skeleton as in Figure~\ref{B2skeleton}.

\begin{figure}[h!]
\includegraphics[width=1 in]{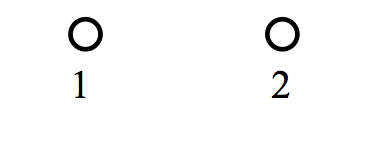}
\caption{The skeleton $\mathcal{D}_*(B_2)$}
\label{B2skeleton}
\end{figure}
thus by Corollary ~\ref{realroots}, 
$$\Phi(B_2)=W\alpha_1\amalg W\alpha_2.$$

The disjoint Weyl orbits of $B_2$, as well as the real roots which are fixed by simple root reflections, are demonstrated in $\mathcal{P}(B_2)$ which is given in Figure~\ref{B2orbit}

\begin{figure}[h!]
\includegraphics[width=2.5 in]{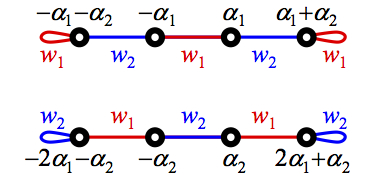}
\caption{Orbit of type $B_2$}
\label{B2orbit}
\end{figure}

The next example is of affine type.  Let \[A_1^{(1)}=\left(\begin{array}{cc} 2 & -2\\ -2 & 2\end{array}\right).\]  Its Dynkin diagram is given in Figure~\ref{affine}.

\begin{figure}[h!]
\includegraphics[width=1 in]{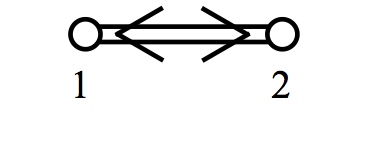}
\caption{Dynkin diagram of type $A_1^{(1)}$}
\label{affine}
\end{figure}
The skeleton $\mathcal{D}_*(A_1^{(1)})$ is given in Figure~\ref{affineskeleton}

\begin{figure}[h!]
\includegraphics[width=1 in]{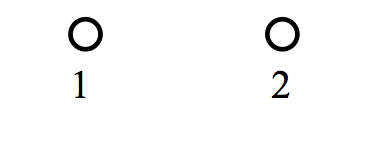}
\caption{The skeleton $\mathcal{D}_*(A_1^{(1)})$}
\label{affineskeleton}
\end{figure}
thus by Corollary ~\ref{realroots}, 
$$\Phi(A_1^{(1)})=W\alpha_1\amalg W\alpha_2.$$  

The disjoint $W$--orbits are demonstrated in $\mathcal{P}(A_1^{(1)})$ (Figure~\ref{affineorbit}).  Here we have chosen to label each root $\beta$ with $(\rho_1(\beta),\rho_2(\beta))$:

\begin{figure}[h!]
\includegraphics[width=4 in]{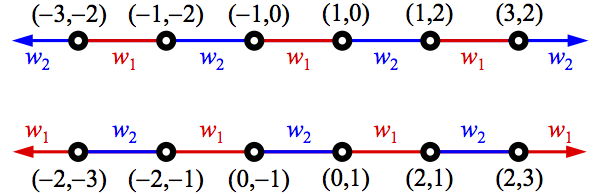}
\caption{$W$--orbit of real roots of $A_1^{(1)}$}
\label{affineorbit}
\end{figure}

Consider the hyperbolic example \[A_2^*=\left(\begin{array}{cc} 2 & -3\\ -3 & 2\end{array}\right).\]  The graph $\mathcal{P}(A_2^*)$ is given in Figure~\ref{hyporbit}.

\begin{figure}[h!]
\includegraphics[width=4 in]{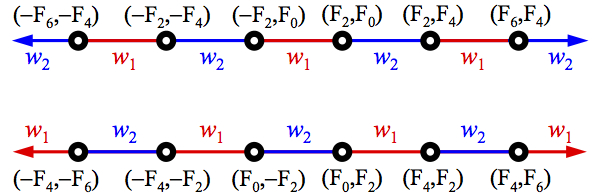}
\caption{$W$--orbit of real roots of $\mathcal{H}(3)$}
\label{hyporbit}
\end{figure}

where $F_n$ denotes the $n$--th Fibonacci number (\cite{F}, \cite{KaM}, \cite{ACP}).

The final two examples correspond to rank 3 root systems.  Consider $H_2^{(3)}$ in ~\cite{CCCMNNP},~\cite{Li}, given by \[H_2^{(3)}=\left(\begin{array}{ccc} 2 & -2 & -1 \\ -2 & 2 & -1 \\ -1 & -1 & 2\end{array}\right)\] with Dynkin diagram as in Figure~\ref{hyp}. 

Its Dynkin diagram (Figure~\ref{hyp}) contains a double edge and double arrow which, when removed, form a connected skeleton (Figure~\ref{conn}).
\begin{figure}[h!]
\includegraphics[width=1 in]{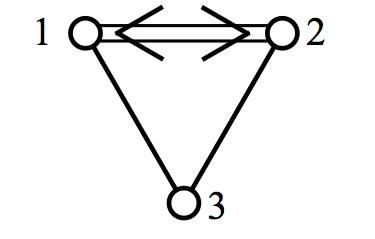}
\caption{Dynkin diagram of type $H_{2}^{(3)}$}
\label{hyp}
\end{figure}

\begin{figure}[h!]
\includegraphics[width=1 in]{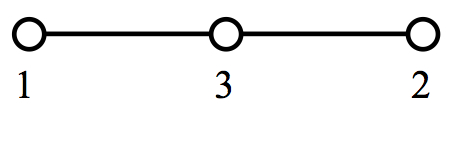}
\caption{The connected skeleton $\mathcal{D}_*(H_{2}^{(3)})$}
\label{conn}
\end{figure}

Thus by Corollary ~\ref{realroots} there is only one $W$--orbit, so we may expect that $\mathcal{P}(H_2^{(3)})$ will be a connected graph.  Indeed this is the case, as shown in Figure ~\ref{fig3} .  As we will prove in Section ~\ref{fixed}, this is one of only two rare examples of a hyperbolic root system on which the Weyl group acts both transitively and simply transitively.

\begin{figure}[h!]
\includegraphics[width=6 in]{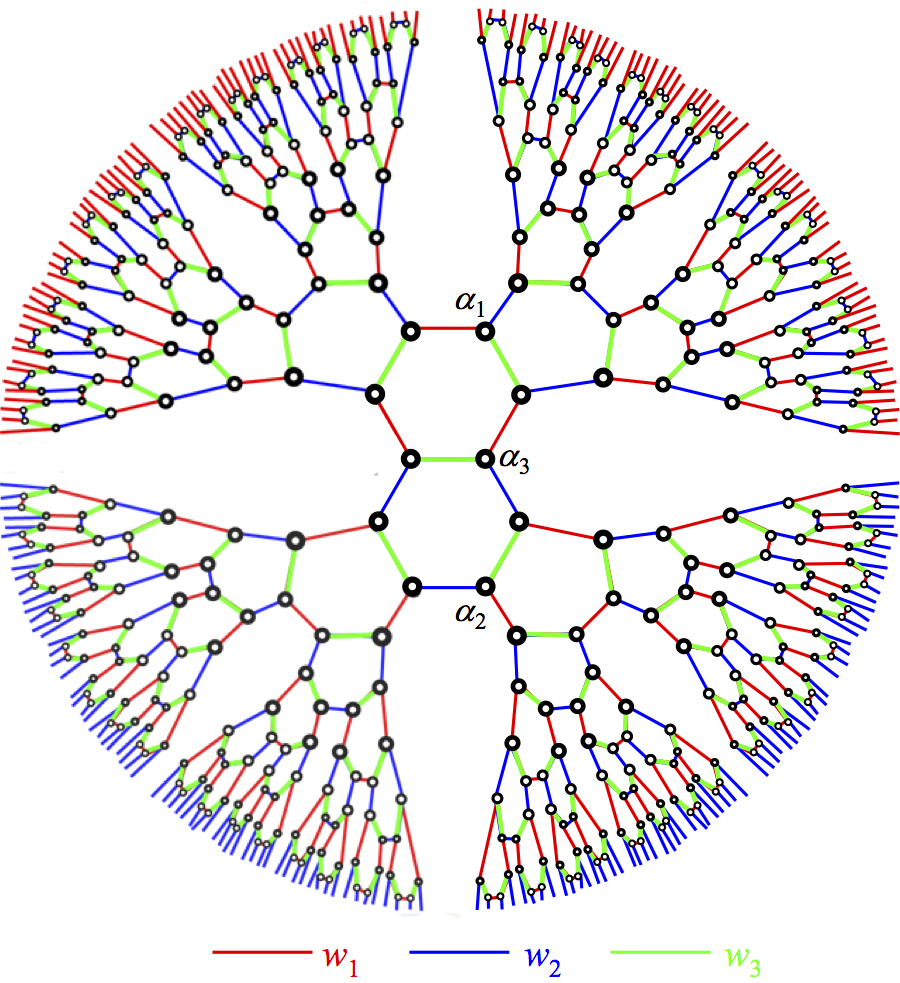}
\caption{The connected $\mathcal{P}-$ graph of hyperbolic root system $H_2^{(3)}$.}
\label{fig3}
\end{figure}

\begin{figure}[t]
\includegraphics[width=6 in]{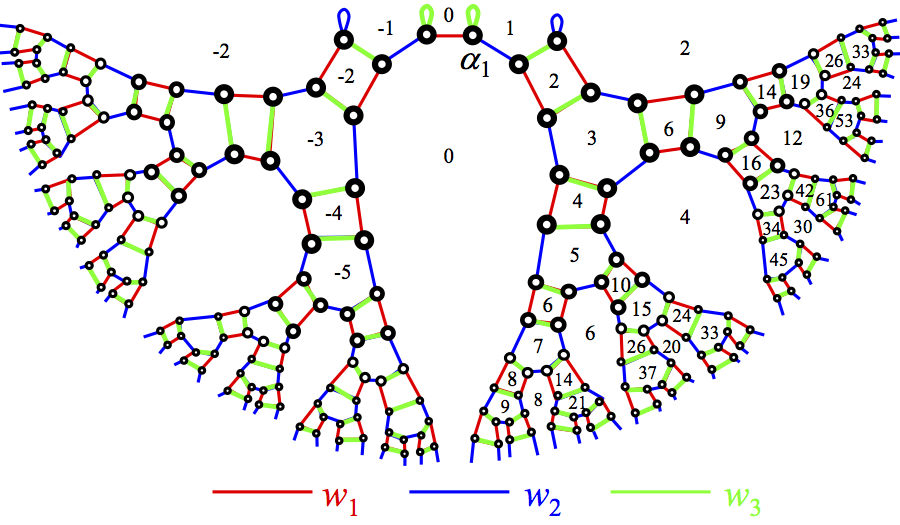}
\caption{Connected component of ${P}(\widehat{A}_1^{(1)})$ corresponding to $W\alpha_1$.  To decode the vertex labelling of the roots, see Remark ~\ref{vertex}.}
\label{orbit1}
\end{figure}

Now consider the Feingold-Frenkel rank 3 hyperbolic root system, 
\[\widehat{A}_1^{(1)}=\left(\begin{array}{ccc} ~2 & -2 & ~0 \\ -2 & ~2 & -1 \\ ~0 & -1 & ~2\end{array}\right).\]  

The Dynkin diagram for $\widehat{A}_1^{(1)}$ is given in Figure~\ref{FF}. The skeleton $\mathcal{D}_*(\widehat{A}_1^{(1)})$ is given in Figure~\ref{FFskeleton}.

\begin{figure}[h!]
\includegraphics[width=1 in]{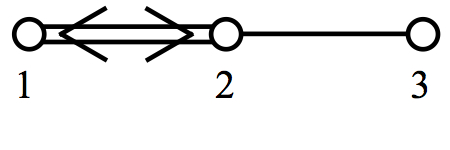}
\caption{Dynkin diagram of type $\widehat{A}_1^{(1)}$}
\label{FF}
\end{figure}

\begin{figure}[h!]
\includegraphics[width=1 in]{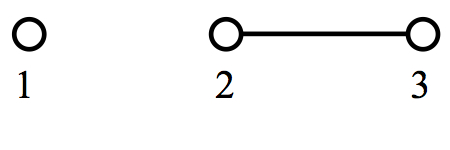}
\caption{The skeleton $\mathcal{D}_*(\widehat{A}_1^{(1)})$}
\label{FFskeleton}
\end{figure}

which implies via Corollary ~\ref{realroots} that the root system $\Phi(\widehat{A}_1^{(1)})$ is the disjoint union, 
$$\Phi=W\alpha_1\amalg W\{\alpha_2,\alpha_3\}.$$   

\begin{figure}[h!]
\includegraphics[width=6 in]{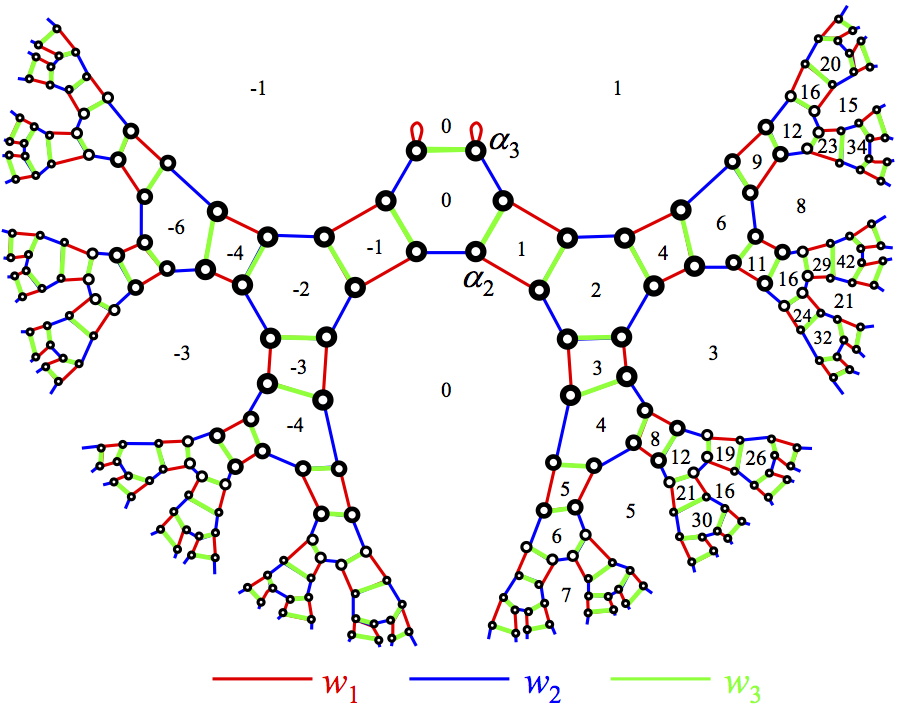}
\caption{Connected component of ${P}(\widehat{A}_1^{(1)})$ corresponding to $W\{\alpha_2,\alpha_3\}$. To decode the vertex labelling of the roots, see Remark ~\ref{vertex}.}
\label{orbit2}
\end{figure}

\begin{rem}\label{vertex}
The vertex labelling in Figures ~\ref{orbit1} and ~\ref{orbit2} is encoded by the labelling of the $4$--, $6$--, and $\infty$--gons of the graph. These polygons correspond to the defining relations of $W$. For instance, the $6$--gons correspond to the relation $(w_2w_3)^3=1$. Given a root $\gamma$  the vertex labeling is given by $(\rho_1(\gamma),\rho_2(\gamma),\rho_3(\gamma))$. \end{rem}

Though the orbit graphs $\mathcal{P}$ for $\widehat{A}_1^{(1)}$ are drawn as if they appear in the Poincar\'e disk, there is no natural embedding of them in the disk. Since they are orbit graphs of real roots, they naturally live on a hyperboloid of 1 sheet which supports the real roots. 

It is an interesting question to compare the combinatorial structure of the $\mathcal{P}$-graph with the combinatorial structure of the action of the Weyl group $W$ on the root lattice. Let $w_i$ denote a set of generators for $W$. Then two real roots $\beta_1$ and $\beta_2$ are connected by a vertex  in $\mathcal{P}$ if one is transformed into the other by a generator $w_i$ of $W$. However, the distance between $\beta_1$ and $\beta_2$ on the root lattice may be arbitrarily large.

Consider for example $A=A_1^{(1)}$ which is a rank 2 subsystem of $\widehat{A}_1^{(1)}$. Then $W=W(A)$ is the infinite dihedral group $W=\mathbb{Z}/2\mathbb{Z}\ast\mathbb{Z}/2\mathbb{Z}$. Let $\beta_0$ be a real root and let $w_n=(w_1w_2)^n\in W$ be some Weyl group element. Then 
$w_n (\beta_0)$ and $(w_2w_n)(\beta_0)$ are adjacent in $\mathcal{P}$. On the other hand, in the root lattice, the Weyl distance $d(w_n \beta_0, (w_2w_n)\beta_0)$ is proportional to $n$.

A similar phenomenon occurs for any hyperbolic Weyl group such that two of its generators $w_i$ and $w_j$ generate a sub--Weyl group $W_{ij}=\mathbb{Z}/2\mathbb{Z}*\mathbb{Z}/2\mathbb{Z}$.

Note that $\mathcal{P}(\widehat{A}_1^{(1)})$ shows the 6 roots in this root system which are fixed by simple reflections:

\begin{enumerate}
\item $w_1$ fixes $\pm \alpha_3$,
\item $w_2$ fixes $\pm (\alpha_1+2\alpha_2+2\alpha_3)$,
\item $w_3$ fixes $\pm \alpha_1$.

\end{enumerate}

\section{Same root lengths in distinct orbits}

For a finite root system associated to a simple Lie algebra, all roots of the same length lie in the same $W$--orbit.  However, in  a Kac--Moody  root system, real roots of the same length may lie in different $W-$orbits.

{\bf Example.} The generalized Cartan matrix of the  Dynkin diagram of rank 4 noncompact hyperbolic type (Figure~\ref{rk4}) is symmetrizable but not symmetric. There are 3 distinct lengths of real roots, but 4 $W$--orbits on real roots. In particular, $\alpha_1$ and $\alpha_3$ have the same root length but lie in different $W$--orbits.

\begin{figure}[h!]
\includegraphics[width=1.85 in]{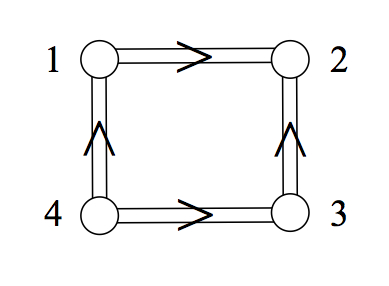}
\caption{Dykin diagram of type $H_{23}^{(4)}$}
\label{rk4}
\end{figure}

Given a symmetric generalized Cartan matrix $A$ of  noncompact hyperbolic type such that real roots of the same length lie in distinct orbits, we show that there is a nonsymmetric generalized Cartan matrix $A'$ with the same Weyl group as $A$ and the same Coxeter matrix as $A$, such that the real roots of different  lengths lie in distinct $W$--orbits. 

In other words, for any Kac--Moody  Weyl group $W$, we show that  there is a symmetrizable Kac--Moody  algebra such that the  lengths of real roots in different orbits are different.

Consider the following symmetric generalized Cartan matrix of noncompact hyperbolic type:
$$A=\left(\begin{array}{ccc} ~2 & -2 & ~0 \\ -2 & ~2 & -1 \\ ~0 & -1 & ~2\end{array}\right)$$
and let 
$$A'=\left(\begin{array}{ccc} ~2 & -1 & ~0 \\ -4 & ~2 & -1 \\ ~0 & -1 & ~2\end{array}\right).$$
Then $A'$ is non symmetric. Also $A$ and $A'$ have different Dynkin diagrams but the same Weyl group $W\cong PGL_2(\mathbb{Z})$ and the same Coxeter matrix
$$\left(\begin{array}{ccc} 1 & \infty & 2  \\ \infty & 1 & 3 \\  2 & 3  & 1 \end{array}\right).$$
Let $\Phi$ be the real roots of $A$ and $\Phi'$ the real roots of $A'$. The 2 different root lengths in $A'$ justify the 2 distinct $W$--orbits of real roots in $W\Phi'$. Since $A$ and $A'$ have the same Coxeter matrix, there are also  2 distinct $W$--orbits of real roots in  $W\Phi$.  However, since $A$ is symmetric, all real roots of $A$ have the same root length.

We will prove that such a phenomenon can only occur if $A$ has rank 2 affine type, or $A$ has rank 3 noncompact hyperbolic type. We will make use of the following result from ([CCCMNNP]).

\begin{proposition}\label{rk3}([CCCMNNP]) A symmetrizable hyperbolic generalized Cartan matrix contains an $A_1^{(1)}$ or $A_2^{(2)}$ proper indecomposable  submatrix if and only if rank $A=3$ and $A$ has  noncompact type.
\end{proposition}

\bigskip
\begin{theorem} Given a symmetric generalized Cartan matrix $A$ of  noncompact hyperbolic type such that real roots of the same length lie in distinct orbits, there is a nonsymmetric generalized Cartan matrix $A'$ with the same Weyl group as $A$ and the same Coxeter matrix as $A$, such that  there is a natural bijective correspondence between the orbit structure $W\Phi'$ of the real roots corresponding to $A'$ and the  orbit structure $W\Phi$ of the real roots corresponding to $A$.  \end{theorem}

\medskip\noindent{\it Proof:} Let $\Phi$ be the real roots of $A$ and $\Phi'$ the real roots of $A'$. If real roots in $\Phi$ of the same root lengths lie in distinct $W$--orbits, then in the skeleton of the Dynkin diagram $\mathcal{D}(A)$, there must be 2 disconnected vertices, say $i$ and $j$. Hence between vertices $i$ and $j$ in the Dynkin diagram, there must have been $\geq 2$ edges. Hence $a_{ij}a_{ji}\geq 4$. Since $A$ has noncompact hyperbolic type, it follows that $A$ contains an $A_1^{(1)}=\left(\begin{array}{cc} ~2 & -2 \\ -2 & ~2  \end{array}\right)$ 
proper indecomposable  submatrix.
Applying Proposition~\ref{rk3}, we see that $A$ must have rank 3 noncompact hyperbolic type. Next we claim that for every Dynkin diagram of rank 3 noncompact hyperbolic type with an $A_1^{(1)}$ subdiagram, there is another Dynkin diagram of rank 3 noncompact hyperbolic type with an $A_2^{(2)}=\left(\begin{array}{cc} ~2 & -1 \\ -4 & ~2  \end{array}\right)$ subdiagram. This is easily checked in the tables in the classification of hyperbolic Dynkin diagrams in ([CCCMNNP]). $\square$

It follows that there is a natural bijective correspondence between the orbit structure $W\Phi'$ of the real roots corresponding to $A'$ and the  orbit structure $W\Phi$ of the real roots corresponding to $A$. 

From [CCCMNNP],  if $A_0$ is a rank 2 proper submatrix of a symmetrizable hyperbolic generalized Cartan matrix $A$ of rank $\geq 4$ then $A_0$ is of finite type. We have the following corollary which follows easily from the tables in the classification of hyperbolic Dynkin diagrams in ([CCCMNNP]).

\begin{corollary} Let $A$ symmetric hyperbolic generalized Cartan matrix of rank $\geq 4$ or symmetrizable and rank $\geq 7$. Then any 2 real roots of the same length lie in the same $W$--orbit.
\end{corollary}

If $A$ has rank 4, 5 or 6 and is not symmetric, then the corollary is false, as the above example indicates.


\section{Fixed points and orbifold graphs associated to the graph $\mathcal{P}$}\label{fixed}

In this section we give sufficient conditions in terms of the generalized Cartan matrix $A$ (equivalently ${\mathcal D}$) for a simple reflection in $W$ to stabilize a simple root.  We thus obtain necessary conditions for the action of $W$ to be simply transitive, and show that the Weyl group for $E_{10}$ acts transitively but not simply transitively on  real roots. We also obtain sufficient conditions for a simple reflection to stabilize a real root that is not simple. 

We begin by recalling the notion of complete graph and we apply the definition to Dynkin diagrams.

\begin{definition}
A Dynkin diagram $\mathcal{D}$ is called \textbf{complete} if it is a complete graph, that is, if every pair of distinct vertices $\mathcal{D}$ is connected by an edge.
\end{definition}

We now establish sufficient conditions for the existence of fixed roots.

\begin{lemma}\label{simply}
Let $A$ be a Cartan matrix (or generalized Cartan matrix), $\mathcal{D}$ its corresponding Dynkin diagram, $\Phi$ the set of real roots with basis $\Pi$ indexed by $I$, and $W$ the Weyl group.  If $A$ contains zeros (equivalently, if $\mathcal{D}$ is not complete), then there exist simple roots in $\Phi$ that are stabilized by simple reflections in $W$, and thus $W$ does not act simply transitively on  $\Phi$.
\end{lemma}
\begin{proof}
We first verify the equivalency in the hypothesis, namely, that if $A$ contains a zero then $\mathcal{D}$ is not a complete graph. Assume $a_{ij}=0$.  By the definition of a Cartan matrix, this implies that $a_{ji}=0$.  Thus the submatrix $A_{(i,j)}$ corresponds to the Cartan matrix for $A_1\times A_1$:
\[A_{(i,j)}=\left( \begin{array}{ccc} 2 & 0\\ 0 & 2\end{array}\right),\]
and this implies vertices $i,j$ in $\mathcal{D}$ are not connected by an edge.  Thus $\mathcal{D}$ is not complete. 

We now prove the existence of fixed simple roots by observing the action of simple roots $w_i$ and $w_j$ on simple roots $\alpha_i$ and $\alpha_j$:
$$w_i\alpha_j=\alpha_j-a_{ji}\alpha_i = \alpha_j,$$
and
$$w_j\alpha_i=\alpha_i-a_{ij}\alpha_j = \alpha_i.$$
Thus $w_j$ fixes $\alpha_k$, and $w_k$ fixes $\alpha_j$.  This is illustrated by the $\mathcal{P}-$graph for $A_1\times A_1$ in Figure~\ref{simple}.

\begin{figure}[h!]
\includegraphics[width=1 in]{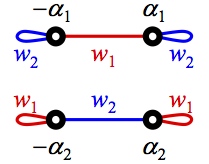}
\caption{Orbit structure of $A_1\times A_1$}
\label{simple}
\end{figure}
The existence of roots fixed by $W$ proves that $W$ does not act simply transitively on  $\Phi$.
\end{proof}

Restating the lemma in its contrapositive form reveals the necessary conditions for a Weyl group to act simply transitively.

\begin{corollary}
If a Weyl group $W$ acts simply transitively on  its corresponding root system $\Phi$, then its associated Dynkin diagram is complete.
\end{corollary}

Since the Dynkin diagrams for $E_{10}$ and $\widehat{A}_1^{(1)}$ are not complete, we have the following immediate consequence of Corollary ~\ref{realroots} and Lemma ~\ref{simply}.

\begin{corollary}\label{noE10}
The Weyl group for $E_{10}$ acts transitively on  its real roots, but not simply transitively. The Weyl group of $\widehat{A}_1^{(1)}$ is neither transitive nor simply transitive on its real roots.
\end{corollary}

These results raise the question, how many hyperbolic root systems have Weyl groups that act both transitively and simply transitively?  In the following corollary, we prove that there are only 2.

\begin{corollary}
Let $\mathcal{D}$ be a Dynkin diagram of hyperbolic type, $A$ its corresponding generalized Cartan matrix, $\Phi$ its corresponding root system and $W$ its Weyl group. Assume that $W$ acts both transitively and simply transitively on  $\Phi$.  Then $\mathcal{D}$ is either $H_2^{(3)}$ or $H_1^{(4)}$ in the classification (\cite{CCCMNNP}, \cite{Li}, \cite{KM}, \cite{Sa}). 
\end{corollary}
\begin{proof}
Since $W$ acts simply transitively on  $\Phi$, by Lemma ~\ref{simply} $\mathcal{D}$ must be complete.  Since $W$ also acts transitively on  $\Phi$, $\mathcal{D}_*$ is connected by Corollary ~\ref{realroots}.  Thus $\Phi$ contains roots of only 1 length, and $A$ is symmetric.   

\textit{Case 1}: $\mathcal{D}$ is comprised of 2 vertices.  Thus,
\[A=\left( \begin{array}{ccc} 2 & -a\\ -a & 2\end{array}\right),\]
where $a\ge3$.  But then $\mathcal{D}$ contains a multiple (bold) edge, which is removed when forming $\mathcal{D}_*$, which is thus disconnected. Hence $W$ does not act transitively on  $\Phi$ by Corollary ~\ref{realroots}, which is a contradiction.

\textit{Case 2}:  $\mathcal{D}$ is comprised of $\ge 3$ vertices. The classification of hyperbolic dynkin diagrams of ranks $3-10$ (maximal) contains 142 symmetrizable diagrams, of which only $17$ diagrams are complete.  Of these diagrams, only $H_2^{(3)}$ and $H_1^{(4)}$ have connected skeletons $\mathcal{D}_*$. 
\end{proof}

Thus Figure ~\ref{fig3} shows the rare case of a hyperbolic root system that consists of only 1 orbit and has no fixed roots.

Note that Lemma ~\ref{simply} applies to both finite and infinite root systems. We now reveal a consequence of the lemma for the finite dimensional theory.

\begin{corollary}
Let $\mathfrak{g}$ be a finite dimensional semisimple Lie algebra of rank $\ge 3$. Then its Weyl group $W$ does not act simply transitively on  its corresponding root system.
\end{corollary}
\begin{proof}
The corresponding Dynkin diagram contains 3 or more vertices. Searching through the classification of finite dimensional simple Lie algebras reveals that $\mathcal{D}$ is not complete, and the result follows.
\end{proof}

Lemma ~\ref{simply} showed sufficient conditions for a simple reflection in $W$ to stabilize a simple root.  The following result provides a criterion for a simple root to stabilize any real root.

\begin{prop} \label{fixedthm} Let $A$ be a generalized Cartan matrix, and let $W(A)$ and $\Phi(A)$ be its associated Weyl group and set of real roots.  Let $A_{(i)}$ denote the $i-$th column vector of $A$,
$$A_{(i)}=\left[\begin{array}{ccc} A_{1i} \\ A_{2i}\\ \vdots \\ A_{li}\end{array}\right].$$
Let $\alpha\in\Phi$.  If $\alpha\cdot A_{(i)}=0$, then $\alpha$ is fixed by $w_i$.

\end{prop}
\begin{proof} 
$$w_i\alpha=\alpha-(\alpha\cdot A_{(i)})\alpha_i=\alpha.$$
\end{proof}

Recall from Section ~\ref{sec4} that the hyperbolic root system of $\widehat{A}_1^{(1)}$ contains 6 fixed roots.  Each of these roots is fixed by a simple root reflection. That is,  $w_1\in\operatorname{Stab}_W(\alpha_3)$ and $w_3\in\operatorname{Stab}_W(\alpha_1)$.  By a direct application of [Ka, Proposition 3.12a], we obtain the following.
\begin{lemma}\label{stab} (From [Ka, Proposition 3.12a]) 
Consider the fixed points $\pm \alpha_1$, $\pm \alpha_3$, and $\pm (\alpha_1+2\alpha_2+2\alpha_3)$ in the hyperbolic root system $\widehat{A}_1^{(1)}$.  We have
\begin{enumerate}
\item $\operatorname{Stab}_W(\alpha_1)=\langle w_3 \rangle$,
\item $\operatorname{Stab}_W(\alpha_3)=\langle w_1 \rangle$,
\item $\operatorname{Stab}_W(\alpha_1+2\alpha_2+2\alpha_3)=\langle w_2 \rangle$.
\end{enumerate}
\end{lemma}

We may use the example of the Feingold--Frenkel rank 3 hyperbolic root system and its graph $\mathcal{P}(\widehat{A}_1^{(1)})$ to indicate the role of fixed points.  If $W$ acted simply transitively on simple roots, the graph $\mathcal{P}(\widehat{A}_1^{(1)})$ would be a single connected component and would coincide with the Cayley graph, dual to the tessellation. Note that this is the case for $\mathcal{P}(H_2^{(3)})$ (Figure ~\ref{fig3}), which coincides with the Cayley graph for the triangle group $(\infty,3,3)$, dual to the tessellation.  However, for $\widehat{A}_1^{(1)}$ there are 2 orbits for $W$ on real roots, hence 2 connected components of the graph $\mathcal{P}(\widehat{A}_1^{(1)})$. If there were no fixed points for $W$ on real roots, these 2 components would both be dual to the tessellation. However, each connected component admits distinct fixed points for the action of $W$ on simple roots. The connected components of $\mathcal{P}(\widehat{A}_1^{(1)})$ are thus orbifold graphs.  For transitive Weyl group actions, the graph $\mathcal{P}$ is dual to the $W$-tessellation of its root space.

\section{Weyl group orbits on imaginary roots}

In this section we examine the structure of the orbits of the Weyl group on imaginary roots.  We use the symmetry properties of the lightcone in the Cartan subalgebra of a hyperbolic Kac--Moody  algebra to deduce some natural properties of the structure of Weyl group orbits on imaginary roots.

Let $\mathfrak{g}$ be a  Kac--Moody  algebra with Cartan subalgebra $\mathfrak{h}$, root space decomposition
$$\mathfrak{g}=\mathfrak{g}^+\ \oplus\ \frak{h}\ \oplus\ \mathfrak{g}^-,$$  
$$\mathfrak{g}^+ =\bigoplus_{\alpha\in\Delta^+}\mathfrak{g}^{\alpha},\ \ 
\mathfrak{g}^- = \bigoplus_{\alpha\in\Delta^-}\mathfrak{g}^{\alpha}$$
and root spaces 
$$\mathfrak{g}^{\alpha}\ =\ \{x\in \mathfrak{g}\mid[h,x]=\alpha(h)x,\ h\in \frak{h}\}.$$
The dimension of the root space $\mathfrak{g}^{\alpha}$ is called the {\it multiplicity} of $\alpha$. We recall that the associated Weyl group acts on the set $\Delta$ of all roots, preserving root multiplicities. 

Now let $\mathfrak{h}$ be the standard Cartan subalgebra of  a symmetrizable hyperbolic Kac--Moody  algebra $\mathfrak{g}$ and let  $\mathfrak{h}_{\mathbb{R}}\subset\mathfrak{h}$ such that 
$\mathfrak{h}=\mathbb{C}\otimes_{\mathbb{R}}\mathfrak{h}_{\mathbb{R}}$ and $\mathfrak{h}_{\mathbb{R}}$ contains the simple roots of $\mathfrak{h}$.
The union $X$ of the sets $w(\mathcal{C})$, for $w\in W$ 
$$X=\bigcup_{w\in W} w(\mathcal{C})$$
is called the {\it Tits cone}. The Weyl group $W$ has a fundamental chamber $\mathcal{C}$ inside $\mathfrak{h}_{\mathbb{R}}$ for its action on $X$. 

When $\mathfrak{g}$ is of hyperbolic type, we have the following description of the union of the closures of the positive and negative Tits cones
$$\overline{X}\cup -\overline{X}=\{h\in \mathfrak{h}_{\mathbb{R}}\mid (h\mid h)\leq 0\},$$
where $(\cdot\mid\cdot)$ is the symmetric bilinear invariant form on $\mathfrak{g}$. The set 
$$\mathcal{L}_{\mathfrak{h}_{\mathbb{R}}}=\{h\in \mathfrak{h}_{\mathbb{R}}\mid (h\mid h)\leq 0\}$$ is called the {\it lightcone} of the Cartan subalgebra $\mathfrak{h}_{\mathbb{R}}$.

As the Adjoint action (and $W$) act by isometries, the length of roots is preserved. As a consequence the action of $W$ preserves the set of imaginary roots of zero squared length and of negative squared length. 

The imaginary roots of negative squared length lie on hyperboloids of a fixed `radius' inside the lightcone while the imaginary roots of zero squared length lie on the boundary of the lightcone. Moreover, for $\mathfrak{g}$ hyperbolic, every lattice point inside the lightcone of $\mathfrak{h}$ corresponds to an imaginary root ([Ka], Ch 5).

If $\alpha$ is a (positive) imaginary root, then by Lemma~\ref{posroot}, the image $W\alpha$  is also positive. Since $W$ preserves $(\alpha\mid\alpha)$ ([Ka], Ch 3), $W\alpha$ lies along the hyperboloid  of radius $(\alpha\mid\alpha)$. We may view the image $W\alpha$ as a  `translation' of $\alpha$ along this hyperboloid.

Thus we have the following obvious but useful characterization of real and imaginary roots in terms of the Weyl group orbits.

Let $\Mg$ be a symmetrizable Lie algebra or Kac--Moody  algebra. Let $\alpha$ be any root. Then
$$\alpha\text{ is real if and only if there exists $w\in W$ such that } w\alpha=-\alpha,$$
$$\alpha\text{ is imaginary if and only if there does not exist $w\in W$ such that } w\alpha=-\alpha.$$

Focusing the attention to sets of roots that lie on the same hyperboloid of fixed squared length, it is possible that these sets contain roots of different multiplicities.  If $\alpha,\beta\in\Delta^{imag}$ are on the same hyperboloid of fixed squared length but $mult(\alpha)\neq mult(\beta)$, then there is no element of $Aut(\mathfrak{g})$ that takes $\alpha$ to $\beta$. In particular there is no element of $W$ that takes $\alpha$ to $\beta$. Hence roots with different multiplicity must lie in different $W$--orbits.




We make the following self--evident but important observation. Let $\mathfrak{g}$ be a hyperbolic Kac--Moody  algebra. Let $\alpha$ and $\beta$ be any pair of imaginary roots of the same squared length that lie inside the fundamental chamber $\mathcal{C}$ for the Weyl group or on its boundary $\partial\mathcal{C}$. Then $\alpha$ and $\beta$ may have the same multiplicity, but lie in distinct orbits under the action of $W$.

It follows that imaginary roots are in the same $W$--orbit, then they must  have the same squared length and the same multiplicity but can't both be in the fundamental chamber for the Weyl group or on its boundary. The following is an easy consequence of the properties of $W$--orbits on the imaginary roots and can be deduced easily from [K] and [M3]. 

\begin{theorem} Let $\mathfrak{g}$ be a Kac--Moody  algebra of hyperbolic type.

(1) For imaginary roots of fixed squared length on a hyperboloid $\mathcal{H}$, there are finitely many $W$--orbits of imaginary roots. 

(2) The cardinality of each such orbit is infinite. 

(3) The number of $W$--orbits on imaginary roots of fixed squared length is bounded above by the number of lattice points in $(\mathcal{C}\cup\partial\mathcal{C})\cap \mathcal{H}$.

(4) The number of $W$--orbits on imaginary roots $\to\infty$ as squared length $\to\infty$.
\end{theorem}

{\it Proof:} Let $\mathcal{H}$  be a  hyperboloid supporting the imaginary roots of some fixed squared length.  Let $\mathcal{S}$ be the set of lattice points  in $(\mathcal{C}\cup\partial\mathcal{C})\cap \mathcal{H}$. Then $|\mathcal{S}|<\infty$. The imaginary roots corresponding to lattice points in $\mathcal{S}$ cannot be in the same $W$--orbit. The roots in $\mathcal{S}$ may or may not have the same multiplicity. If they have the same multiplicity, they may be in the same $W$--orbit. If the roots in $\mathcal{S}$ do not all have the same multiplicity, then they will fall into finitely many $W$--orbits. Thus $W\mathcal{S}$ is a union of finitely many infinite disjoint sets. This proves (1) and (2) and (3) and (4) follow easily. $\square$

\end{document}